\documentclass[12pt]{amsart}
\usepackage{a4wide, amssymb, bbm, calligra, mathrsfs}
\usepackage[2cell,arrow,curve,matrix]{xy}
\UseHalfTwocells
\UseTwocells

\def\id{{\sf Id}}

\def\scgr{{\frak S}{\frak C}{\frak G}}

\def\A{{\mathbb A}}
\def\B{{\mathbb B}}

\def\D{{\mathbb D}}

\def\G{{\mathbb G}}

\def\I{{\mathbb I}}

\def\M{{\mathbb M}}

\def\P{{\mathbb P}}
\def\Q{{\mathbb Q}}
\def\R{{\mathbb R}}
\def\S{{\mathbb S}}

\def\Z{{\mathbb Z}}

\def\Ho{{\bf Ho}}

\def\bA{{\bf A}}

\let\al\alpha

\let\d\partial

\let\x\times

\let\then\Rightarrow

\let\then\Rightarrow

\def\cok{\mathop{\sf Coker}\nolimits}

\def \ker{\mathop{\sf Ker}\nolimits}
\def \ext{\mathop{\sf Ext}\nolimits}

\def\xto#1{\xrightarrow[]{#1}}

\newtheorem{Pro}{Proposition}
\newtheorem{Le}[Pro]{Lemma}
\newtheorem{The}[Pro]{Theorem}
\newtheorem{Co}[Pro]{Corollary}
\theoremstyle{definition}
\newtheorem{De}[Pro]{Definition}
\theoremstyle{remark}

\newtheorem{Rem}[Pro]{Remark}

\def\ab{_{\mathrm{ab}}}

\def\al{{\alpha}}

\def\mhom{{\bf Hom}}
\def\type{{\sf k}}

\def\Ty{{\bf \Gamma AB}}

 \def\ta{{\frak T}}

\def\ab{{\frak A}{\frak b}}

\def\xyma{\xymatrix@M.7em}

\let\x\times

\begin{document}

\title{ On abelian 2-categories and derived 2-functors}
\author[T. Pirashvili]{Teimuraz  Pirashvili}
\address{
Department of Mathematics\\
University of Leicester\\
University Road\\
Leicester\\
LE1 7RH, UK} \email{tp59-at-le.ac.uk}
\thanks{Research was partially supported by the GNSF Grant
ST08/3-387}

\maketitle

\hfill{\emph{Dedicated to the memory  of Prof. V. K. Bentkus}}

\section{Introduction} We will assume that the reader is
familiar with the work of Mathieu Dupont on abelian 2-categories
\cite{dupont} (called $2$-abelian ${\sf Gpd}$-category in
\cite{dupont}). We do not recall this quite technical definition
and give only general remarks and few examples.

The main difference between abelian categories and abelian
2-categories is that in abelian 2-categories we have not only
objects and morphisms but we have also 2-cells or tracks between two
parallel arrows. To be more precise abelian 2-categories  are first
of all groupoid enrich categories.

Recall that a \emph{groupoid} is a small category such that all
morphisms are isomorphisms. For a groupoid $\G$ and an object $x\in
\G$ we let $\pi_0(\G)$ and $\pi_1(\G,x)$ be the set of connected
components of $\G$ and the group of automorphisms of $x$ in $\G$
respectively.

The prototype of abelian categories is the category of abelian
groups $\ab$. A similar role in the two dimensional world is plaid
by the 2-category $\scgr$ of symmetric categorical groups. The
notion of a symmetric categorical group is a categorification of a
notion of abelian group. More precisely, let $(\A, +,0,a,l,r,c)$ be
a symmetric monoidal category, where $+:\A\x\A \to \A$ is the
composition law, $0$ is the neutral element, $a$ is the associative
constrants, $c$ is the commutativity constrants and $l:\id \to
0+\id$ and $r:\id \to \id+0$ are natural transformations satisfying
well-known properties \cite{working}. We will say that $\A$ is a
\emph{symmetric categorical group} or \emph{Picard category}
provided $\A$ is a groupoid and for any object $x$ the endofunctor
$x+:\A\to \A$ is an equivalence of categories. It follows that
$\pi_0(\G)$ is an abelian group and the endofunctor $x+:\A\to \A$
yields an isomorphism $\pi_1(\G,0) \to \pi_1(\G,x)$ of abelian
groups. In what follows we will write $\pi_1(\G)$ instead of
$\pi_1(\G,0)$.

Symmetric categorical groups form the 2-category $\scgr$ which is
the prototype for abelian 2-categories. In an abelian 2-category
$\ta$ for any two object $A$ and $B$  the hom groupoid
$\mhom_\ta(A,B)$ is in fact a symmetric categorical group. Moreover
$\ta$  has kernels and cokernels in the 2-dimensional sense and
satisfies exactness properties.

For any object $A$ of an abelian 2-category $\ta$, the composition of morphisms
equips the symmetric categorical group $\mhom(A,A)$ with a
multiplication.  This two structure form a mathematical object
called 2-ring. Here a $2$-ring (called also Ann-category
\cite{quang}, or categorical ring \cite{cat_rings}) is a
categorification of the version of a ring and it consists of a
symmetric categorical group equipped with ''multiplication''
satisfying Laplaza coherent axioms \cite{lapl}. This notion with
different name already presents in the pioneering work of Takeuchi
\cite{tak}.

Large class of examples of abelian $2$-categories are given by the
2-categories of $2$-modules over $2$-rings \cite{dupont},
\cite{vincent}, \cite{china1}.

Any category can be considered as a groupoid enrich category with
trivial tracks. Hence the theory of groupoid enrich categories
generalizes usual category theory. In a sense the theory of abelian
2-categories also generalizes the theory of abelian categories, but
in a more tricky way. Let $\bf A$ be an abelian category, we wish to
associate to it an abelian 2-category. The first idea which comes in
mind is to consider $\bA$ as a groupoid enrich category with trivial
tracks. However in this way we never obtain an abelian 2-category
except the trivial case $\bA=0$ \cite{dupont}. More interesting is
the following construction. Let $\bA$ be an abelian category and
consider the 2-category $\bA^{[1]}$ consisting of arrows
$A=(A_1\xto{a} A_0)$ considered as chain complexes concentrated in
dimensions $1$ and $0$. Then morphisms in $\bA^{[1]}$ are just chain
maps and tracks are just chain homotopies. However $\bA^{[1]}$ is
usually NOT an abelian 2-category, except the case when $\bA$ is
semi-simple \cite{dupont}. To solve the dilemma, let us assume that
$\bA$ is an abelian category with enough projective objects and
consider the full 2-subcategory $\bA^{[1]}_c$ consisting of arrows
$A=(A_1\xto{a} A_0)$ with projective $A_0$. Then $\bA^{[1]}_c$ is an
abelian 2-category \cite{ab-2-ab} and this is a way how abelian
categories should be considered as a part of the theory of abelian
2-categories. To support this point of view let us observe that any
abelian 2-category $\ta$ defines the derived category $\D(\ta)$,
which is triangulated category in the usual sense \cite{verdier}. It
is easy to observe that the  derived category of $\bA^{[1]}_c$
triangulated equivalent to the classical derived category $\D(\bA)$
of $\bA$.

It is interesting to compare $\bA^{[1]}$ and $\bA^{[1]}_c$ from the
point of view of usual category theory. The underlying category of
$\bA^{[1]}$ is abelian, while the underling category of
$\bA^{[1]}_c$  has no kernels nor cokernels in general. However from
the point of view of 2-dimensional algebra $\bA^{[1]}_c$ is much
nicer. This is no accident and the reader familiar with homological
and homotopical algebra recognize the role of cofibrant objects.

Of course the dual construction works as well. Namely, if $\bA$ is
an abelian category with enough injective objects then the
2-category $\bA^{[1]}_f$ consisting of arrows $A=(A_1\xto{a} A_0)$
with injective $A_1$ is an abelian 2-category. Assuming now that
$\bA$ has enough projective and injective objects. So we could
perform both constructions $\bA^{[1]}_c$ and $\bA^{[1]}_f$ and they
looks quite different from the point of view of the usual category
theory but they are 2-equivalent abelian 2-categories (actually
using ''butterflies'' \cite{butterflies} these constructions can be
unify and get an abelian 2-category starting from an arbitrary
abelian category satisfying some smallness conditions).

It is maybe worthwhile to say that the notion of 2-equivalence is a
2-dimensional analogue of the notion of equivalence of categories.
The set of objects is NOT invariant under equivalence of categories.
Similarly the underlain category of an abelian 2-category is NOT
invariant under 2-equivalences as we have seen in the examples
$\bA^{[1]}_c$ and $\bA^{[1]}_f$. Now we give more examples of the
same phenomena.

Assume $R$ is a ring. The category $\bA$ of left $R$-modules is an
abelian category, so we can consider the corresponding abelian
2-category $\bA^{[1]}_c$. On the other hand $R$ can be seen as a
discrete 2-ring. So we have another abelian 2-category, namely the
2-category of 2-modules over $R$. These 2-categories are very
different from the point of view of the usual category theory but
they are 2-equivalent abelian 2-categories (see Corollary
\ref{r-2r}).

It is well-known and is absolutely trivial that the category of
abelian groups is isomorphic to the category of modules over the
ring of integers. Probably a similar fact is also true for symmetric
categorical groups. The role of integers should be played by the
free symmetric categorical group with one generator equipped with
appropriate multiplication. After my  suggestion Vincent Schmitt
worked on this and related problems, but unfortunately his work
\cite{vincent} in this direction is  unfinished.

However, if one considers the problem not up to isomorphisms but up
to 2-equivalences, then it can be solved quite easily. Namely, one
can consider the following symmetric categorical group $\Phi$.
Objects of $\Phi$ are integers. If $n\not =m$ then there is no
morphism from $n$ to $m$, $n,m\in \Z$. The group of automorphisms of
$n$ is the cyclic group of order two with generator $\epsilon_n$,
$n\in \Z$. The monoidal functor in $\Phi$ is induced by the addition
of integers. The associativity and unite constrants are identity
morphisms, while the commutativity constrant $n+m\to m+n$ equals to
$nm\epsilon_{n+m}$. We will see that $\Phi$ plays the same role in
$\scgr$ as the abelian group of integers plays in the category of
abelian groups in the following sense: $\Phi$ has a natural $2$-ring
structure induced by the multiplication of integers (all
distributivity constrants being trivial) and the 2-category of
2-modules over $\Phi$ in fact is 2-equivalent to the 2-category
$\scgr$ (see Proposition \ref{pi-r}).

These facts can be easily deduced  from two theorems proved below.
The first one is  a 2-dimensional analogue of the Gabriel-Mitchel
theorem (see Theorem \ref{2-g-m} below) while the second theorem
claim that abelian 2-category $\scgr$ (as well as the abelian
2-category of 2-modules over a 2-ring) has enough projective
objects. In fact it has also enough injective objects. The result on
projective and injective objects first was  proved in \cite{rpo-in}
and this works should be considered as an extended version of it. We
also added small sections on resolutions and derived 2-functors,
following to \cite{mamuka_adams}, \cite{ext-b-v} and \cite{2-ch}. In
particular we develop the theory of secondary ext objects in abelian
2-categories and we show that the cohomology of $2$-groups can be
described via such ext.

\section{Gabriel-Mitchel theorem} We start with some definitions.
Following to \cite{dupont} an \emph{additive $2$-functor} from an
abelian $2$-category $\ta$ to another abelian $2$-category $\ta_1$
is a pseudo-functor which on hom-s is a morphism of symmetric
monoidal categories. A 2-functor $F:\ta\to\ta'$ is called a
\emph{$2$-equivalence} between abelian 2-categories if the functors
$\ta(A,A')\to \ta'(FA,FA')$ are equivalences of groupoids for all
objects $A,A'$ of $\ta$ and each object $B$ of $\ta'$ is equivalent
to some object of the form $F(A)$. If $\ta$ is an abelian
$2$-category, we let $\Ho(\ta)$ be the category which has the same
objects as $\ta$, while for objects $A$ and $B$ we have
$\hom_{\Ho(\ta)}(A,B)=\pi_0{\bf Hom}_\ta(A,B)$. This category is
known as the \emph{homotopy category} of $\ta$. A morphism $f:A\to
B$ in an abelian 2-category $\ta$ is called \emph{faithful} provided
for all object $X$ the induced functor $f^X:\ta(X,A)\to \ta(X,B)$ is
faithful. An object $A$ in an abelian 2-category $\ta$ is called
\emph{injective} provided for any faithful morphism $F:\S_1\to \S_2$
the induced map
$$\hom_{\Ho(\ta)}(S_2,A)\to \hom_{\Ho(\ta)}(S_1,A)$$
is surjective. We will say that an abelian 2-category $\ta$ has
enough injective objects provided for any object $\A$ there exist a
faithful morphism $\A\to \S$ with injective object $\S$.

Dually, a morphism $f:A\to B$ in an abelian 2-category $\ta$ is
called \emph{cofaithful} provided for all object $X$ the induced
functor $f_X:\ta(A,X)\to \ta(B,X)$ is faithful. An object $A$ is
called \emph{projective} provided for any cofaithful morphism
$F:S_1\to S_2$ the induced map
$$\hom_{\Ho(\ta)}(A,S_1)\to \hom_{\Ho(\ta)}(A,S_2)$$
is surjective. We will say that an abelian 2-category $\ta$ has
enough projective objects provided for any object $A$ there exist a
cofaithful morphism $S\to A$ with projective object $S$.

A \emph{coproduct} of the family $A_i$, $i\in I$ of objects in an
abelian 2-category $\ta$ is an object $A=\bigoplus_{i\in I} A_i$
equipped with maps $\mu_i:A_i\to A$, $i\in I$ such that for any
object $X\in\ta$ the induced morphism
$$
\mu^*:\mhom_\ta(A,X)\to \prod_{i\in I}\mhom_\ta(A_i,X)
$$
is an equivalence of groupoids. In this case $A$ is coproduct of the
family $A_i$, $i\in I$ in the category $\Ho(\ta)$ as well. By
duality we can also talk on products. Observe that if $A_1$ and
$A_2$ are objects in an abelian 2-category $\ta$ then there exist a
product $A=A_1\times A_2$ which is also a coproduct. We will say
that an abelian 2-category $\ta$ has coproducts if for any family of
objects $A_i$, $i\in I$ there exists coproduct $\oplus_{i\in I}
A_i$. Dually for products. It is obvious that coproduct of
projective objects is projective and product of injective objects is
injective.

It is easy to show that the abelian 2-category $\scgr$ has all
coproducts and products. Similarly for the 2-category of 2-modules
over 2-rings.

We will say that an object $G$ is a \emph{generator} of an abelian
2-category $\ta$ provided for any object $S$ there is a diagram
$$
\xymatrix{ A\ar[r]^{f}\rruppertwocell<12>^{0}{^\hskip-2ex\alpha}
&B\ar[r]^{g} &C }
$$such that the relative cokernel of this diagram is equivalent to
$S$ and objects $A$, $B$, $C$ are coproducts of $G$. An object $S$
is \emph{small} if $\mhom_\ta(\S,-)$ preserves coproducts. The
reader familiar with the corresponding notion in the classical world
maybe find this definition ad hoc. However, it is known (see for
example pp. 52-53 in \cite{bass}) that an object $G$ of an abelian
category $\A$ is a generator if and only if any object of $\A$ is
isomorphic to the cokernel of a map $G_1\to G_0$, where $G_1 , G_2$
are coproducts of $G$.

In an abelian 2-category $\ta$ any morphism has a kernel and
cokernel in the sense of 2-categories. The cokernel of the morphism
$X\to 0$ is denoted by $\Sigma X$, while the kernel of the morphism
$0\to X$ by $\Omega X$. It is well-known that $\Omega^2=0$ and
$\Sigma^2=0$. Objects $X$ is called \emph{discrete} (resp.
\emph{codiscrete}, or \emph{connected}) if $\Omega X=0$ (resp.
$\Sigma X=0$). The full subcategory of $\Ho(\ta)$ formed by discrete
(resp. codiscrete) objects is an abelian category denoted by ${\sf
Dis}(\ta)$ (resp. ${\sf Codis}(\ta)$). It is a remarkable fact that
${\sf Dis}(\ta)$ and ${\sf Codis}(\ta)$ are equivalent categories.

The following is a 2-dimensional analogue of the classical
Gabriel-Mitchel's theorem (p.54 in \cite{bass}). To state it recall
that if $\ta$ is an additive 2-category and $M$ is an object in
$\ta$ then one has the 2-ring ${\bf End}_\ta(M):={\bf Hom}_\ta(M,
M)$ (compare \cite{dupont},\cite{vincent}) and for any object $X$,
the symmetric categorical group ${\bf Hom}(M, X)$ is a right
2-module over $h(X)={\bf Hom}_\ta(M, X)$. In this way we get a
2-functor $h$ from $\ta$ to the 2-category of 2-modules over ${\bf
End}_\ta (M)$.

\begin{The} \label{2-g-m}
Let $\ta$ be an abelian 2-category with arbitrary coproducts. If $M$
is a small projective generator in $\ta$ then 2-functor $h$ from
$\ta$ to the category of right categorical modules over the
categorical ring $R={\bf End}(M)$ is a 2-equivalence of abelian
2-categories.
\end{The}

\begin{proof} Our argument is almost literary follows to pp. 54-55 in
\cite{bass}. The fact that if $R$ is a 2-ring, then $R$ considered
as a right $R$-module is a small projective generator is an easy
consequence of the 2-dimensional Yoneda lemma. Assume now that $\ta$
satisfies the conditions of the theorem. We need to establish two
facts. Firstly, for any objects $X,Y$ in $\ta$ the functor
$$h_{X,Y}:{\bf Hom}_\ta(X,Y)\to {\bf Hom}_R(hX,hY)$$
is an equivalence of categories and secondly every 2-module $Z$ is
equivalent to some $h(X)$.

We start to verify the first assertion. We fix $Y$ and consider
$h_{X,Y}$ as a natural transformation $\alpha:T\to S$, where
$T(X)={\bf Hom}_\ta(X,Y)$ and $S(X)={\bf Hom}_R(hX,hY)$. Observe
that if $X=M$ then $\alpha(X)$ is an equivalence of categories by
Yoneda and because of smallness assumption it is also an equivalence
for $X$ isomorphic to a coproduct of $M$. Since hom-s are left exact
and ${\bf Hom}_\ta(M,-)$ is an exact functor (due to projectivity of
$M$) it follows that both 2-functors $T$ and $S$ are left exact as
well. Since $M$ is a generator it follows that $\alpha(X)$ is an
equivalence of categories for any $X$.

To see the second assertion, we use the fact that $R$ is a generator
in the 2-category of $R$-modules, hence we can write $Z$ as a
relative cokernel of a diagram
$$
\xymatrix{
R^{(I_1)}\ar[r]^{f}\rruppertwocell<12>^{0}{^\hskip-2ex\alpha}
&R^{(I_2)}\ar[r]^{g} &R^{(I_3)}}
$$
where $R^{(I)}$ denotes the coproduct of $I$-copies of $R$. Since
$$\mhom_R((hM)^{(I_i)}, ((hM)^{(I_j)})\cong \mhom_R((h(M^{(I_i)}), (h(M^{(I_j)}))$$
we see that the diagram is equivalent to one which comes from a
similar diagram in $\ta$ by applying the 2-functor $h$ and again by
exactness of $h$ we see that $Z$ is equivalent to one of the form
$h(X)$.
\end{proof}
 As an immediate consequence of this theorem we see that in
 terminology \cite{dupont} any 2-abelian ${\sf Gpd}$-category which
 posses a small projective generator and arbitrary coproducts is
 automatically good 2-abelian ${\sf Gpd}$-category. Another
 important consequence is the following.

\begin{Co}\label{r-2r} Let $R$ be a classical ring and $\A$ be the abelian
category of classical $R$-modules. Then the abelian 2-category
$\A^{[1]}_c$ is equivalent to the 2-category of 2-modules over $R$
considered as a discrete 2-ring.
\end{Co}
\begin{proof} Consider the following object $R[0]=(0\to R)$ of the abelian
2-category $\A^{[1]}_c$. It is obvious that $R[0]$ is a small and
projective in the abelian 2-category $\A^{[1]}_c$, to show that it
is also a generator take any object $A=(A_1\to A_0)$ in $\A^{[1]}_c$
and consider   a projective resolution $P_*\to A_1$  in the
classical sense. Based on the description  of cokernels of morphism
of $\A^{[1]}_c$ given in \cite{ab-2-ab} it is obvious that  $A$ is
equivalent to the relative cokernel of the diagram
$$
\xymatrix{ P_1[0]\ar[r]\rruppertwocell<12>^{0}{^\hskip-2ex\alpha}
&P_0[0]\ar[r] & A_0[0] }
$$
where $\alpha$ is the trivial track.  So we can apply Theorem
\ref{2-g-m} and the fact that the 2-ring ${\bf
Hom}_{\A^{[1]}_c}(P[0],P[0])$ is isomorphic to $R$ considered as a
discrete 2-ring to finish the proof.
\end{proof}

\section{Morita Equivalence}

Two 2-rings $\R$ and $\S$ are called \emph{Morita equivalent}
provided the 2-categories of 2-modules over $\R$ and $\S$ are
2-equivalent. Based on the tensor product developed by Vincent
Schmitt \cite{vincent} lots of materials on Morita contexts
\cite{bass} have 2-dimensional analogues. Details left to the
interesting readers. Here we restrict ourself only with the
following consequence of our version of the Gabriel-Mitchel theorem.

\begin{The} $\R$ and $\S$ are Morita equivalent iff there exist a
small projective generator $P$ in the 2-category of 2-modules over
$\R$ such that two rings $\S$ and $\mhom_\R(P,P)$ are equivalent.
\end{The}

\begin{proof} It suffices to observe that under 2-equivalences small projective generators corresponds
to small projective generators and then use Theorem \ref{2-g-m}.
\end{proof}

\section{2-chain complexes}\label{2-complex}
In this section we fix an abelian 2-category $\ta$. In particular
for morphisms $f,g:A\to B$ there are defined morphisms $f+g:A\to B$
and $-f:A\to B$.

A \emph{2-chain complex} $(A_*,d,\partial)$ in $\ta$ is a diagram of
the form
$$
\xymatrix{ \cdots\ar[r]\rrlowertwocell<-12>_{0}{}
&A_{n+2}\ar[r]|-{d_{n+1}}\rruppertwocell<12>^{0}{^\partial_n}
&A_{n+1}\ar[r]|-{d_n}\rrlowertwocell<-12>_{0}{_{}{\hskip1.2em\d_{n-1}}}
&A_n\ar[r]|-{d_{n-1}}\rruppertwocell<12>^{0}{^{}} &A_{n-1}\ar[r]
&\cdots }
$$
i.~e., a sequence of objects $A_n$, maps $d_n:A_{n+1}\to A_n$ and
tracks $\partial_n:d_nd_{n+1}\then0$, $n\in\Z$, such that for each
$n$ the tracks
$$
\xymatrix{d_{n-1}d_nd_{n+1}\ar@{=>}[r]^{d_{n-1}\d_n}&d_{n-1}0\ar@{=>}[r]^\equiv&0}
$$
and
$$
\xymatrix{d_{n-1}d_nd_{n+1}\ar@{=>}[r]^{\partial_{n-1}d_{n+1}}&0d_{n+1}\ar@{=>}[r]^\equiv&0}
$$
coincide.

For any 2-chain complex $(A_*,d,\partial)$ and any integer $n$,
there is a well-defined object called $n$-th homology ${\bf
H}_n(A_*)$ of $A_*$ (see p. 138 \cite{dupont}). Following to
\cite{mamuka_adams} we call ${\bf H}_*(A_*)$ the \emph{secondary
homology} of $A_*$. Let us recall the definition. Let ${\bf
Z_n}(A_*)$ be the relative kernel $\ker(d_n,\partial_{n-1})$. Then
we get a natural morphism $k: A_{n+1}\to {\bf Z_n} $ and by the
definition the secondary homology ${\bf H}_n$ is the relative
cokernel of the diagram
$$
\xymatrix{
\bA_{n+2}\ar[r]^{d_{n+1}}\rruppertwocell<12>^{0}{^\hskip-2ex\alpha}
&A_{n+1}\ar[r]^{k} & {\bf Z_n}}
$$
In the abelian 2-category $\scgr$ we put $$H^n(A_*):=\pi_0({\bf
H}_n(\A_*)$$ These groups are known as \emph{Takeuchi-Ulbrich
homology} \cite{2-ch}.

One of the main properties of the secondary homology is that it
associates to an extension of 2-chain complexes a long 2-exact
sequence of secondary homologies \cite{2-ch}. This implies also
usual exactness of the long exact sequence of Takeuchi-Ulbrich
homologies. The following easy lemma on 2-exact sequences will be
useful in the section on derived 2-functors.

\begin{Le}\label{2z} If $$\cdots \to A_n\xto{f_n} B_n\xto{g_n} C_n\xto{h_n} A_{n-1}\to\cdots$$
is a 2-exact sequence of symmetric categorical groups, then
$$\pi_0(A_n)\cong \pi_1(A_{n-1})$$ If $\pi_1A_n=0$ (resp.
$\pi_0A_{n-1}=0$)  then $B_n\cong Ker(h_n)$ (resp. $C\cong
\cok(f_n)$) .
\end{Le}

If $A_*$ and $B_*$ are 2-chain complexes in a 2-abelian category
$\ta$, one can introduce a new 2-chain complex $({\bf Hom}(A_*,B_*),
d,
\partial)$ whose $n$-th component is the product $${\bf
Hom}(A_*,B_*)_n=\prod_{p}{\bf Hom}_\ta(A_p, B_{n+p})$$ The boundary
map
$$d:{\bf Hom}(A_*,B_*)_n\to {\bf Hom}(A_*,B_*)_{n-1}$$ is given by
$$(df)_{p}=df_p+(-1)^{p+1}f_{p-1}d$$
with following track
$$dd(f)\then d^2f+fd^2\then 0$$ where the first track comes
from the distributivity law, while the second is the sum
$f^*(\partial)+f_*(\partial)$.

The objects of the symmetric categorical group ${\bf H}_0({\bf
Hom}(A_*,B_*))$ are known as 2-chain maps from $A_*$ to $B_*$, while
morphisms of ${\bf H}_0({\bf Hom}(A_*,B_*))$ form tracks of 2-chain
maps. We let ${\bf 2Chain}(\ta)$ be the 2-category of 2-chain
complexes, 2-chain maps and tracks of 2-chain maps. The
corresponding homotopy category $\Ho({\bf 2Chain}(\ta))$ is the
homotopy category of 2-chain complexes. Thus a 2-chain map is a pair
$(f,\phi)$, where $f=(f_n)$ is a sequence of morphisms $f_n:A_n\to
B_n$ and $\phi=(\phi_n)$ is the sequence of tracks $f_n d_n\then d_n
f_{n+1}$ satisfying an obvious coherence condition, similarly tracks
$(f,\phi)\then (f',\phi')$ are also pairs $(h,\psi)$, where
$h=(h_n)$ is the sequence of maps $h_n:A_n\to B_{n+1}$ and
$\psi=(\psi_n)$ is the sequence of tracks $d_nh_n+f_n\then
g+h_{n-1}d_{n-1}$ satisfying  an obvious coherence condition.

The secondary homology ${\bf H}_n$ defines an additive 2-functor
$${\bf H}_n:{\bf 2Chian}(\ta)\to \ta$$
In particular the Tackeuchi-Ulbrich homology is homotopy invariant,
meaning that it factors trough the homotopy category $\Ho({\bf
2Chain}(\ta))$ of 2-chain complexes.

Observe that if $T:\ta\to \ta'$ is an additive 2-functor between
abelian 2-categories it induces a well-defined additive 2-functor
${\bf 2Chian}(\ta)\to {\bf 2Chian}(\ta')$, which is two-dimensional
analogue on the degreewise action of an additive functor on chain
complexes.

A 2-chain map $(f_*, \phi_*):A_*\to B_*$ is called \emph{weak
equivalence} provided the induced morphism in secondary homology
${\bf H_*}(A_*)\to {\bf H_*}(B_*)$ is an equivalence. We let
$\D(\ta)$ be the localization of the category $\Ho({\bf
2Chain}(\ta))$ with respect to weak equivalences. Here we emphasise
that the category $\D(\ta)$ exists provided $\ta$ posses countable
coproducts and have enough projective objects. This can be proved
essentially by repeating the argument of Samson Saneblidze given in
\cite{samson}. It is not surprising that the category $\D(\ta)$ when
it exists has a canonical triangulated category structure which is
induced by the following mapping cone construction. Let $(f_*,
\phi_*):(A_*,d^A,\partial^A) \to (B_*,d^B,\partial^B) $ be a 2-chain
map. Define 2-chain complex $(C_*, d,\partial)$ by
$$C_n=A_{n-1}\oplus B_n$$
$$d_n=\begin{pmatrix}-d^A&f\\0&d^B\end{pmatrix}$$
while $\partial d^2\then 0$ is the composite of the canonical track
(coming from the distribution) $d^2\then \begin{pmatrix}d^2&-df+fd\\
0&d^2\end{pmatrix}$ and by the track $\begin{pmatrix} \partial^\A&
\phi\\ 0&\partial^B\end{pmatrix}$. The triangulated category
$\D(\ta)$ is called the derived category of an abelian 2-category
$\ta$.

We claim that this construction is not only an analogue of the
classical construction of derived categories of abelian categories
but also  generalizes it. In fact, if $\bA$ is an abelian category
with enough projective objects then the derived category $\D(\bA)$
in the classical sense and the derived category $\D(\bA^{[1]}_c)$ in
the new sense are triangulated equivalent. The equivalence is given
by the following functors. Define a triangulated functor
$\D(\bA^{[1]}_c)\to \D (\bA)$ by the ''total complex'' construction
(see for example Example 2. 11 in \cite{mamuka_adams}). The functor
in opposite direction is given as follows. If $X_*$ is a chain
complex in $\bA$, first we have to replace it by a weak equivalent
one which consists of projective objects (here one needs to assume
that $\bA$ has countable coproducts) and then apply the functor
$P\mapsto P[0]$ to obtain a 2-chain complex. Here $P[0]=(0\to P)$ is
an object in $(\bA)^{[1]}_c$. One sees that these constructions are
mutually quasi-inverse functors.

It is a remarkable fact that the construction of  a triangulated
category from an abelian 2-category is intimately related to the
derived category construction in the brave new algebra. Namely, if
$R$ is a 2-ring and $H(R)$ is the corresponding ring-spectrum
constructed in \cite{cat_rings} then the derived category of
2-modules is triangulated equivalent to the derived category of
$H(R)$ in the brave new-algebra sense. The proof of this fact will
appear in a forthcoming paper \cite{2-robi}. In particular, there is
a ring spectrum $\Lambda=H(\Phi)$, which is characterized by the
following properties: $\pi_0(\Lambda)=\Z$, $\pi_1(\Lambda)=\Z/2\Z$,
$\pi_i(\Lambda)=0$, if $i\not =0,1$ and the first Postnikov
invariant (as a ring spectrum) of $\Lambda$ is the generator of the
third MacLane (i.e. topological Hochschild) cohomology ${\bf
HML}^3(\Z,\Z/2\Z)=\Z/2\Z$. Then $\D(\scgr)$ and $\D(\Lambda)$ are
triangulated equivalent.

\section{Resolutions and Derived 2-functors} Let $\ta$ be an abelian 2-category with enough projective
objects and $A$ be an object in $\ta$. Following to
\cite{mamuka_adams} a \emph{left $A$-augmented 2-chain complex} is a
2-chain map $(\epsilon, \hat{\epsilon}):A_*\to A$ of 2-chain
complexes, where $A$ is considered as a 2-chain complex concentrated
in dimension 0 with trivial differentials $d=0, \partial=0$ and
$(A_*,d,\partial)$ is a 2-chain complex with $A_n=0$, $n<0$.
Moreover $\partial_n$ equal to the identity track for $n<0$
\cite{mamuka_adams}. A left $A$-augmented 2-chain complex
$(\epsilon, \hat{\epsilon}):A_*\to A$ is called \emph{left
resolution} provided $(\epsilon, \hat{\epsilon})$ is a weak
equivalence. A left resolution is called \emph{projective
resolution} provided $A_n$ is projective for all $n\geq 0$. This
notion is a particular case of a more general notion given in
\cite{mamuka_adams} and corresponds to the case when (in the
notations \cite{mamuka_adams}) $\bf b$ coincides with the class of
all projective objects.

Then we can summarize some results of \cite{mamuka_adams} in the
following proposition.
\begin{Pro} Let $\ta$ be an abelian 2-category with enough projective objects and let
${\bf PR }$ be the full 2-subcategory of the 2-category ${\bf
2Chian}(\ta)$ consisting of projective resolutions. Then the
2-functor given by taking the $0$-th secondary homology defines a
2-equivalence $${\bf PR}\to\ta$$
\end{Pro}

For the reader convenient we recall how to construct projective
resolutions. Take an object $A\in \ta$ and choose a projective
object $P_0$ together with a cofaithful morphism $\epsilon: P_0\to
A_0$. Let $(K_0,i_0, \kappa_0)$ be the kernel of $\epsilon$, where
$i_0:K\to P_0$ is a morphism and $\kappa_0:\epsilon \circ i_0\then
0$ is a track. Choose a projective object $P_1$ together with a
cofaithful morphism $\epsilon_1: P_1\to K_0$. Now we set
$d_0=i_0\epsilon_1:P_1\to P_0$ and $\hat{\epsilon}
=\epsilon_1^*(\kappa_0)$ which is a track $\epsilon\circ d_0\then
0$. Let $(K_1,i_1,\kappa_1)$ be the relative kernel $\ker(d_0;
\hat{\epsilon})$. Choose a projective object $P_2$ together with a
cofaithful morphism $\epsilon_2:P_2\to K_1$ and set
$d_1=i_1\epsilon_2:P_2\to P_1$, $\partial_0=\epsilon_2^*(\kappa_1)$
and continue this way.

Let $\verb"P"$ be the full 2-subcategory of $\ta$ which consists of
projective objects. Then for any additive 2-functor $T:\verb"P"\to
\scgr$ one obtains a well-defined additive 2-functors ${\sf
L}_n(T):\ta\to \scgr$ (called the \emph{secondary left derived
2-functors}) by
$${\bf L}_n(T) (A):={\bf H}_n(T(P_*))$$
where $P_*$ is a projective resolution of $A$. If one takes the
Takeuchi-Ulbrich homology instead and get the Takeuchi-Ulbrich left
derived functors, which are denote by $L_nT$, $n\geq 0$. So by the
very definition one has $\pi_0({\bf L}_n(T) )=L_nT$. It follows then
that $\pi_1 ({\bf L}_n(T) )=L_{n+1}T$ (see Remark 3.2 in
\cite{2-ch}).

In the following proposition we summarize the basic properties  of
derived functors.
\begin{Pro}\label{L-tv} i) The secondary derived functors form a system of $\d$-functors. More precisely,
if
$$
\xymatrix{ A\ar[r]^{f}\rruppertwocell<12>^{0}{^\hskip-2ex\alpha}
&B\ar[r]^{g} &C }
$$
is an extension in $\ta$ (see Definition 2.2 in \cite{ext-b-v}),
then the sequence of symmetric categorical groups
$$\cdots \to {\bf L}_{n+1}T(C)\to {\bf L}_{n}T(A)\to {\bf L}_{n}T(B)\to
{\bf L}_{n}T(C)\to \cdots$$ is 2-exact. Furthermore we have the
following exact sequence of abelian groups
$$\cdots \to {L}_{n+1}T(C)\to {L}_{n}T(A)\to {L}_{n}T(B)\to
{L}_{n}T(C)\to \cdots$$

ii) If $A$ is any object of $\ta$, then $${\bf L}_{n}T(A)=\begin{cases} 0, & n<-1 \\
\Sigma {\bf L_0}T(A),&n=-1\end{cases}$$ Hence $L_nT=0$ if $n<0$.

iii) If $P$ is projective, then
$${\bf L}_{n}T(P)\cong \begin{cases}  T(P), &
n=0\\ \Omega T(P)& n=1\\ 0, & n\not =-1, 0,1\end{cases}$$
Furthermore $${ L}_{n}T(P)\cong \begin{cases} \pi_i(TP) &n=0,1\\ 0&
n\not = 0,1\end{cases}$$

 iv)
Assume  $T_n:\ta\to \scgr$, $n\in \Z$  is a system of $\d$-functors
such that $T_n=0$, if $n<-1$. If for any projective $P$ one has
$T_n(P)=0$ for $n>1$ and $\pi_1 T_1(P)=0$, then there exist a
natural equivalence of 2-functors
$${\bf L}_nT_0\cong T_n, \ n\in \Z$$
\end{Pro}

\begin{proof} i) Easily follows from the fact that if $f_*:P_*\to
P_*'$ is a morphism of projective resolutions, which induces $f$ in
the zeroth homology, then the mapping cone of $f$ is a projective
resolution of $C$. ii) The fact ${\bf L}_n=0$ for $n\leq -2$ is
obvious. Applying Lemma \ref{2z} to the 2-exact sequence from i) we
first get $\pi_0({\bf L}_{-1}(C))=\pi_1({\bf L}_{-2}(C))=0$ and then
$\pi_0({\bf L}_0(C))=\pi_1({\bf L}_{-1}(C)).$ The statement iii) is
clear. To show iv) observe that the argument in ii) shows that
$T_{-1}(A)\cong \Sigma T_0(A)$. Next, we choose an extension
$$
\xymatrix{ B\ar[r]^{f}\rruppertwocell<12>^{0}{^\hskip-2ex\alpha}
&P\ar[r]^{g} &A }
$$
with projective $P$. If one applies the long exact sequence we get
equivalence $$T_{n+1}A\cong T_nB, \ \ n \geq 2$$ and the following
2-exact sequence
$$0\to T_2(A)\to T_1(B)\to T_1(P)\to T_1(A)\to T_0(B)\xto{f_*} T_0(P)$$
We claim that the canonical morphism $T_1A\to ker(f_*)$ is an
equivalence. By 2-exactness it suffice to show that the induced map
$\pi_1T_1(A)\to \pi_1T_0(B)$ is a monomorphism, but this follows
from the fact that $T_1P$ is discrete. Hence all functors $T_n$ can
be express in terms of $T_0$ and we are done.

\end{proof}

Observe that theory of left derived functors have obvious analogue
for right derived 2-functors, both for covariant and contravariant
case. Of course in covariant case we have to use injective
resolutions instead of projective ones.

\begin{Pro}\label{R-tv} i)
The secondary right derived functors of a 2-functor $T:\ta\to \scgr$
form a system of $\delta$-functors. More precisely, if
$$
\xymatrix{ A\ar[r]^{f}\rruppertwocell<12>^{0}{^\hskip-2ex\alpha}
&B\ar[r]^{g} &C }
$$
is an extension in $\ta$ , then the sequence of symmetric
categorical groups
$$\cdots \to {\bf R}^{n}T(A)\to {\bf R}^{n}T(B)\to {\bf R}^{n}T(C)\to
{\bf R}^{n+1}T(A)\to \cdots$$ is 2-exact. Moreover for any object
$A$ of $\ta$, one has
$${\bf R}^{n}T(A)=\begin{cases} 0, & n<-1 \\
\Omega {\bf R^0}T(A),&n=-1\end{cases}$$
 Furthermore, if $I$ is injective, then
$${\bf R}^{n}T(I)\cong \begin{cases}  T(I), &
n=0\\ \Sigma T(I)& n=1\\ 0, & n\not =-1, 0,1\end{cases}$$

ii) Assume  $T^n:\ta\to \scgr$, $n\in \Z$  is a system of
$\delta$-functors such that $T^n=0$, if $n<-1$. If for any injective
$P$ one has $T^n(I)=0$ for $n>1$ and $\pi_0 T^1(I)=0$, then there
exist a natural equivalence of 2-functors
$${\bf R}^nT^0\cong T^n, \ n\in \Z$$
\end{Pro}

\section{Applications to ${\bf Ext}$}

  In particular one can take
the 2-functors ${\bf Hom}(A,-)$ or  ${\bf Hom}(-,B)$ and get the
secondary derived 2-functors. As in the classical case these two
approach gives equivalent objects. Moreover we will show that in
dimension 1 we recover the ${\bf Ext}^1$ from \cite{ext-b-v}.

To start with we let ${\bf Ext}^n_\ta(A,-)$, $n\in \Z$ be the
secondary right derived functors of the 2-functors ${\bf
Hom}_\ta(A,-)$.   We use ${Ext}^n_\ta(A,-)$ for Takeuchi-Ulbrich
derived functors. Of course we are assuming that $\ta$ has enough
injective objects. Then by the dual of Proposition \ref{L-tv} we
have:

\begin{Pro} i) There are natural equivalences
$${\bf Ext}_\ta^n(A,B)\cong \begin{cases} 0 &n\leq -2\\
\hom_{\Ho(\ta)}(A,\Omega B)=\hom_{\Ho(\ta)}(\Sigma A,B)& n=-1\\ {\bf
Hom}_\ta(A,B), & n=0\end{cases}$$ In dimension $-1$ it is understood
that the abelian group $\hom_{\Ho(\ta)}(A,\Omega B)$ is considered
as a discrete symmetric categorical group. Thus
$$ {Ext}_\ta^n(A,B)\cong \begin{cases} 0 &n\leq -2\\
\hom_{\Ho(\ta)}(A,\Omega B)=\hom_{\Ho(\ta)}(\Sigma A,B)& n=-1\\
{\hom}_{\Ho(\ta)}(A,B), & n=0\end{cases}$$

ii) If $I$ is an injective object in $\ta$, then ${\bf
Ext}_\ta^n(A,I)=0$ for $n>1$ and ${\bf Ext}_\ta^1(A,I)$ is a
connected symmetric categorical group, with $$\pi_1({\bf
Ext}_\ta^1(A,I))\cong \hom_{\Ho(\ta)}(A,B)$$ Thus
$${Ext}_\ta^1(A,I)=0$$

iii) Let
$$
\xymatrix{ B\ar[r]^{f}\rruppertwocell<12>^{0}{^\hskip-2ex\alpha}
&C\ar[r]^{g} &D }
$$
be an extension in $\ta$. Then the sequence
$$\cdots \to {\bf Ext}_\ta^n(A,B)\to {\bf Ext}_\ta^n(A,C)\to {\bf Ext}_\ta^n(A,D)\to
{\bf Ext}_\ta^{n+1}(A,B)\to \cdots$$ is a 2-exact sequence of
symmetric categorical groups. Moreover the sequence
$$\cdots \to {Ext}_\ta^n(A,B)\to {Ext}_\ta^n(A,C)\to {Ext}_\ta^n(A,D)\to
{Ext}_\ta^{n+1}(A,B)\to \cdots$$ is an exact sequence of abelian
groups.

\end{Pro}

Assuming now  $\ta$ has enough projective objects as well. Then we
have

\begin{Pro} i) If $P$ is an projective object in $\ta$, then ${\bf
Ext}_\ta^n(P,B)=0$ for $n>1$ and ${\bf Ext}_\ta^1(P,B)$ is a
connected symmetric categorical group, with $$\pi_1({\bf
Ext}_\ta^1(P,B))\cong \hom_{\Ho(\ta)}(P,B)$$ Thus ${
Ext}_\ta^n(P,B)=0$ for $n>0$.

ii) Let
$$
\xymatrix{ A\ar[r]^{f}\rruppertwocell<12>^{0}{^\hskip-2ex\alpha}
&B\ar[r]^{g} &C }
$$
be an extension in $\ta$. Then the sequence
$$\cdots \to {\bf Ext}_\ta^n(C,D)\to {\bf Ext}_\ta^n(B,D)\to {\bf Ext}_\ta^n(A,D)\to
{\bf Ext}_\ta^{n+1}(C,D)\to \cdots$$ is a 2-exact of symmetric
categorical groups. Moreover the sequence
$$\cdots \to {Ext}_\ta^n(C,D)\to {Ext}_\ta^n(B,D)\to {Ext}_\ta^n(A,D)\to
{Ext}_\ta^{n+1}(C,D)\to \cdots$$ is an exact sequence of abelian
groups.

iii) The right derived functors of the 2-functor ${\bf
Hom}_\ta(-,B)$ are isomorphic to ${\bf Ext}_\ta^n(-,B)$.

iv) The categorical group ${\bf Ext}_\ta^1(A,B)$ is isomorphic to
the symmetric categorical group of extensions as it is defined in
\cite{ext-b-v}.
\end{Pro}
\begin{proof}
 i) Since $P$ is projective, the 2-functor ${\bf Hom}_\ta(A,-)$
respects relative exact sequences and the result follows. ii) Take
an injective resolution $I^*$ of $D$, then ${\bf Hom}_\ta(C,I^*)\to
{\bf Hom}_\ta(B,I^*)\to {\bf Hom}_\ta(A,-)$ is an extension of
2-chain complexes and the result follows from \cite{2-ch}. iii) By
i) and ii) the result follows from iv) of Proposition \ref{L-tv} and
finally iv) follows from Corollary 11.3 in \cite{ext-b-v} and the
formula for first derived functor obtained in the proof of iv)
\ref{L-tv}.

\end{proof}

\section{The homotopy category of $\scgr$}
For symmetric categorical groups $\S_1$ and $\S_2$ we have a
groupoid (in fact a symmetric categorical group \cite{ext-b-v})
${\bf Hom}(\S_1,\S_2)$. It follows from the result of \cite{sinh}
that the 2-category $\scgr$ is 2-equivalent to the 2-category of
two-stage spectra (see also Proposition B.12 in \cite{HS}). Hence we
can use the classical facts of algebraic topology to study $\scgr$.
Let $\Ty$ be the category of triples $(A,B,a)$ where $A$ and $B$ are
abelian groups and $$a\in \hom(A/2A,B)=\hom(A,\ _2B)$$ where
$_2B=\{b\in B \mid 2b=0\}$. A morphism $(A,B,a)\to (A_1,B_1,a_1)$ is
a pair $(f,g)$ where $f:A\to A_1$ and $g:B\to B_1$ are
homomorphisms, such that $a_1f=ga$. The functor
$$\type:\Ho(\scgr)\to \Ty$$ is defined by
$$\type(\S):=(\pi_0(\S),\pi_1(\S),k_\S)$$
where $\S$ is a symmetric categorical group and $k_\S$ is the
homomorphism induced by the commutativity constrants in $\S$.
\begin{Pro} For any symmetric categorical groups $\S_1$ and
$\S_2$ one has a short exact sequence of abelian groups
\begin{equation}\label{sinh0}
0\to \ext(\pi_0(\S_1),\pi_1(\S_2))\to \pi_0({\bf Hom}(\S_1,\S_2))\to
\Ty(\type(\S_1),\type(\S_2))\to 0
\end{equation} Furthermore one has also an isomorphism of abelian
groups
\begin{equation}\label{sinh1}
\pi_1({\bf Hom}(\S_1,\S_2))\cong
\hom(\pi_0(\S_1),\pi_1(\S_1))\end{equation}
\end{Pro}
\begin{proof} The second isomorphism is obvious, while the first one
is Proposition 7.1.6 in \cite{htype}.
\end{proof}
We see that the both categories $\Ho$ and $\Ty$ are additive and the
functor
$$\type:\Ho\to \Ty$$
is additive. In fact it is a part of a linear extension of
categories (see Lemma 7.2.4 and Theorem 7.2.7 in \cite{htype}). It
follows from the properties of linear extensions of categories
\cite{ah} that the  functor $\type$ is full, reflects isomorphisms,
is essentially surjective on objects and it induces a bijection on
the isomorphism classes of objects. Moreover the kernel of $\type$
(morphisms which goes to zero) is a square zero ideal of $\Ho$.
Hence, for a given object $\A$ of the category $\Ty$ we can choose a
symmetric categorical group $K(\A)$ such that $\type(K(\A))=\A$.
Such object exist and is unique up to equivalence. Moreover, for any
morphism $f:\A\to \B$ we can choose a morphism of symmetric
categorical groups $K(f):K(\A)\to K(\B)$, such that $\type(K(f))=f$.
The reader must be aware that the assignments $\A\to H(\A)$,
$f\mapsto K(f)$ does NOT define a functor $\Ty\to \Ho$. Having in
mind relation with spectra, the construction $K$ for the objects of
the form $(A,0,0)$ coincides with Eilenberg-MacLane spectrum and in
general case is consistent with Definition 7.1.5 in \cite{htype}.

\section{Projective objects in $\scgr$}

In this section we prove the following theorem. Let us recall that
the symmetric categorical group $\Phi$ was defined in the
introduction.

\begin{The}\label{pi-sym} i) The symmetric categorical group $\Phi$ is a
small projective generator in $\scgr$. In particular $\scgr$ has
enough projective objects. Moreover, any projective object is
equivalent to a coproduct of copies of $\Phi$.

ii) For any 2-ring $R$ the right module $R$ is a small projective
generator of the abelian 2-category of 2-modules, in particular the
abelian 2-category of 2-modules has enough projective objects.
\end{The}

The statement on 2-modules is a direct consequence of Yoneda lemma
for 2-categories. The statement on $\scgr$ is a consequence of Lemma
\ref{adj} proved below.

Thanks to \cite{kv} a morphism $f$ in $\scgr$ is faithful (resp.
cofaithful) if underlying functor is faithful (resp. essentially
surjective). Recall also that \cite{kv} a morphism $F:\S_1\to \S_2$
in $\scgr$ is essentially surjective if it is epimorphism on
$\pi_0$, while a morphism $F:\S_1\to \S_2$ in $\scgr$ is faithful if
it is monomorphism on $\pi_1$. We can develop same sort of language
in the category $\Ty$. A morphism $f=(f_0,f_1)$ in $\Ty$ is
\emph{essentially surjective} if $f_0$ is epimorphism of abelian
groups. Moreover an object $\P$ in $\Ty$ is \emph{projective} of for
any essentially surjective morphism $f:\A\to \B$ in $\Ty$ the
induced map
$$\Ty(\P,\A)\to \Ty(\P,\B)$$ is surjective.

It is clear that a morphism $F:\S_1\to \S_2$ of symmetric
categorical groups is \emph{essentially surjective} iff $\type(F):
\type(\S_1)\to \type(\S_2)$ is so in $\Ty$. For an abelian group $M$
we introduce two objects in $\Ty$:
$$l(M):=(M,M/2M,id_{M/2M}),$$
$$M[0]=(M,0,0).$$

\begin{Le}\label{adj} i) If $M$ is an abelian group and
$\A=(A_0,A_1,\al)$ is an object in $\Ty$, then one has the following
functorial isomorphism of abelian groups
$$\Ty(l(M),\A)=\hom(M,A_0).$$

ii) An object $\P$ is projective in $\Ty$ iff it is isomorphic to
the object of the form $l(P)$ with free abelian group $P$.

iii) $\Phi$ is a projective object in $\scgr$ and any projective
object in $\scgr$ is equivalent to a  coproduct of $\Phi$.

iv) The 2-category of symmetric categorical groups have enough
projective objects.
\end{Le}
\begin{proof} i) Assume $f=(f_0,f_1):l(M)\to \A$ is a morphism in
$\Ty$. So $f_0:M\to A_0$ and $f_1:M/2M\to A_1$ are homomorphisms of
abelian groups and the following diagram is commute
$$\xymatrix{M/2M\ar[d]^{\hat{f_0}}\ar[r]^\id&M/2M\ar[d]^{f_1}\\
A_0/2A_0\ar[r]_{\al}&A_1}$$ Here $\hat{f_0}$ is induced by $f_0$. It
follows that $f_1$ is completely determined by $f_0$. This proves
the result.

ii) Let $P$ be a free abelian group and let $\A\to\B$ be an
essentially surjective morphism in $\Ty$. Thus $A_0\to B_0$ is an
epimorphism. It follows that $\hom(P,A_0)\to\hom(P,B_0)$ is
epimorphism as well, hence by the virtue of i) the map
$\Ty(l(P),\A)\to \Ty(l(P),\B)$ is surjective. Thus $l(P)$ is
projective in $\Ty$. Conversely, assume $\P=(P_0,P_1,\pi)$ is a
projective object in $\Ty$. We claim that $P_0$ is a free abelian
group. In fact it suffice to show that it is a projective object in
the category $\ab$ of abelian groups. Take any epimorphism of
abelian groups $f_0:A\to B$ and any homomorphism of abelian groups
$g_0:P_0\to B$. We have to show that $g_0$ has a lift to $A$.
Observe that $f=(f_0,0):A[0]\to B[0]$ is essentially surjective in
$\Ty$ and $g=(g_0,0):\P\to B[0]$ is a well-defined morphism in
$\Ty$. By assumption we can lift $g$ to a morphism $\tilde{g}:\P\to
A[0]$, It is clear that $\tilde{g}=(\tilde{g_0},0)$ for some
$\tilde{g_0}:P_0\to A$. Clearly $g_0=f_0\circ \tilde{g_0}$. It
follows that $P_0$ is a free abelian group. Hence $l(P_0)$ is a
projective object in $\Ty$. By i) the identity map defines a
canonical morphism $i=(\id_{P_0},i_1):l(P_0)\to \P$, which obviously
is essentially surjective in $\Ty$. Since $\P$ is projective it
follows that there exist a morphism $p=(\id_{P_0},p_1):\P\to l(P)$
such that $i\circ p=\id_\P$. Thus we have a commutative diagram
$$\xymatrix{P_0/2P_0\ar[r]^{\pi}\ar[d]^\id &P_1\ar[d]^{p_1}\\
P_0/2P_0\ar[r]^{\id}\ar[d]^\id& P_0/2P_0\ar[d]^{i_1}\\
P_0/2P_0\ar[r]^{\pi}&P_1}$$ with $i_1p_1=\id_{P_1}$. It follows that
$p_1$ and $i_1$ are mutually inverse isomorphisms of abelian groups.
Hence $p:\P\to l(P)$ and $l:l(P)\to \P$ are mutually inverse
isomorphisms in $\Ty$.

iii) First of all observe that $\type(\Phi)=l(\Z)$. Hence our
assertion is equivalent to  the following one: For any free abelian
group $P$ the symmetric categorical group $K(l(P))$ is projective
symmetric categorical group and conversely, if $\S$ is a projective
symmetric categorical group then $\pi_0(\S)$ is a free abelian group
$\S$ is equivalent to $H(l(\pi_0(\S)))$. To prove it, let $F:\S_1\to
\S_2$ be an essentially surjective morphism of symmetric categorical
groups and $G:K(l(P))\to \S_2$ be a morphism of symmetric
categorical groups. Apply the functor $\type$ to get morphisms
$\type(F):\type(\S_1)\to \type(\S_2)$ and $\type(G):l(P)\to
\type(\S_2)$ in $\Ty$. Since $\pi_0(F):\pi_0(\S_1)\to\pi_0(\S_2)$ is
an epimorphism of abelian groups it follows that
$\type(F):\type(\S_1)\to \type(\S_2)$ is an essentially surjective
morphism in $\Ty$. Since $P$ is a free abelian group $l(P)$ is
projective in $\Ty$ by ii). Thus we can lift $\type(F)$ to get a
morphism $\hat{g}:l(P)\to \type(\S_1)$ such that $\type(F)\circ
\hat{g}=\type(G)$ holds in $\Ty(l(P),\type(\S_2))$. Since
$P=\pi_0(K(l(P))$ is free abelian group the Ext-term in the exact
sequence (\ref{sinh0}) disappears and we get the isomorphism
\begin{equation}\label{sinh2}\pi_0({\bf Hom}(K(l(P)),\S_i))\cong
\Ty(l(P),\type(\S_i)), \ \ i=0,1
\end{equation}
Take a morphism $L:K(l(P)\to \S_1$ of symmetric categorical groups
which corresponds to the morphism $\hat{g}:l(P)\to \type(\S_1)$. By
our construction one has an equality $\type(FL)=\type(G)$ in
$\Ty(l(P), \type(\S_2))=\pi_0({\bf Hom}(H(l(P)),\S_2))$. Thus the
classes of $FL$ and of $G$ in $\pi_0({\bf Hom}(K(l(P)),\S_1))$ are
the same. Hence there exist a track from $FL$ to $G$. This shows
that $K(l(P))$ is a projective symmetric categorical group.
Conversely assume $\S$ is a projective symmetric categorical group.
Since $\S$ and $K(\type(\S))$ are equivalent, it follows that
$K(\type(\S))$ is also projective. We claim that $\type(\S)$ is
projective in $\Ty$. In fact take any essentially surjective
morphism $f=(f_0,f_1): \A\to\B$ and any morphism $g:\type(\S)\to \B$
in $\Ty$. Then $K(f):K(\A)\to K(\B)$ is essentially surjective in
$\scgr$. Hence for $K(g):K(\type(\S))\to K(\B)$ we have a morphism
$\tilde{G}:K(\type(\S))\to K(\A)$ and a track $K(f)\circ
\tilde{G}\to K(g)$. Thus $K(f)\circ \tilde{G}= K(g)$ in $\pi_0({\bf
Hom}(K(\type(\S)),K(\B)))$. Now apply the functor $\type$ to get the
equality $f\circ \type(\tilde{G})=g$, showing that $\type(\S)$ is
projective in $\Ty$. Hence $\type(\S)$ is isomorphic to $l(P)$ for a
free abelian group $P$. Thus $\S$ and $K(l(P)$ are equivalent.

iv)  Let $\S$ be a symmetric categorical group. Choose a free
abelian group $P$ and an epimorphism of abelian groups $f_0:P\to
\pi_0(\S)$. By Lemma \ref{adj} it has a unique extension to a
morphism $f=(f_0,f_1):l(P)\to \type(\S)$ which is essentially
surjective. Since $P$ is a free abelian group, we have the
isomorphism (\ref{sinh2}), which show that there exist a morphism of
symmetric categorical groups $K(l(P)) \to \S$ which realizes $f_0$
on the level of $\pi_0$. Clearly this morphism does the job.
\end{proof}

\begin{Pro}\label{pi-r} The 2-category of symmetric categorical groups is
2-equivalent to the category of right categorical modules over the
categorical ring $\Phi$.
\end{Pro}
\begin{proof} Since $\Phi$ is a small projective generator the 2-category of
symmetric categorical groups is 2-equivalent to the category of
right categorical modules over the categorical ring ${\bf
Hom}(\Phi,\Phi)$. Observe that we have an obvious morphisms of
$2$-rings $\Phi\to {\bf Hom}(\Phi,\Phi)$ which sends the object $n$
to the endomorphism of $\Phi$ which on objects is given by $x\mapsto
nx$. It remains to show that this morphism of 2-rings is an
equivalence. But this easily follows from the exact sequence
\ref{sinh0}.
\end{proof}

\section{Injective objects}

In this section we prove the following result.

\begin{The}\label{inj}  The abelian category $\scgr$ as well as the abelian 2-category of 2-modules over a 2-ring
have enough injective objects.
\end{The}

These are just part iv) and v) of Lemma \ref{adji} proved below.

A morphism $f=(f_0,f_1)$ in $\Ty$ is \emph{faithful} provided $f_1$
is injective and an object $\I=(I_0,I_1,\iota)$ of $\Ty$ is
\emph{injective} if for any faithful morphism $f:\A\to \B$ in $\Ty$
the induced map
$$\Ty(\B,\I)\to \Ty(\A,\I)$$ is surjective.
It is clear that a morphism $F:\S_1\to \S_2$ of symmetric
categorical groups is faithful iff $\type(F): \type(\S_1)\to
\type(\S_2)$ is faithful in $\Ty$. For an abelian group $M$ we
introduce two objects in $\Ty$:
$$r(M)=(_2M,M,id_{_2M}),$$
$$M[1]=(0,M,0).$$

\begin{Le}\label{adji} i) If $M$ is an abelian group and
$\A=(A_0,A_1,\al)$ is an object in $\Ty$, then one has the following
functorial isomorphism of abelian groups
$$\Ty(\A,r(M))=\hom(A_1,M).$$

ii) An object $\Q$ is an injective object in $\Ty$ iff it isomorphic
to the object of the form $r(Q)$ with divisible abelian group $Q$.

iii) For any divisible abelian group $Q$ the symmetric categorical
group $K(r(Q))$ is injective. Conversely, if $\S$ is an injective
categorical group then $\pi_1(\S)$ is a divisible abelian group and
$\S$ is equivalent to $K(r(\pi_1(\S)))$.

iv) The 2-category of symmetric categorical groups have enough
injective objects.

v) Let $\R$ be a categorical group. Then the category of categorical
right $\R$-modules have enough injective objects.

\end{Le}
\begin{proof} i) Assume $g=(g_0,g_1):\A\to r(M)$ is a morphism in $\Ty$ i in
$\Ty$. So $g_0:A_0\to _2M$ and $g_1:A_1\to M$ are homomorphisms of
abelian groups and we have a commutative diagram:
$$\xymatrix{A_0\ar[r]^{\al}\ar[d]_{g_0}&_2A_1\ar[r]^i\ar[d]_{\bar{g_1}}&A_1\ar[d]^{g_1}\\
_2M\ar[r]^{\id}& _2M\ar[r]^j& M}
$$
where $\bar{g_1}$ is induced by $g_1$ and $i,j$ are inclusions. It
follows that $g_0$ is completely determined by $g_1$ and the result
follows.

ii) Let $Q$ be a divisible abelian group and let $\A\to\B$ be a
faithful morphism in $\Ty$. Thus $A_1\to B_1$ is a monomorphism.
Since $Q$ is an injective object in $\ab$ it follows that
$\hom(B_1,Q)\to\hom(A_1,Q)$ is an epimorphism of abelian groups. So
by i) the map $\Ty(\B, r(Q))\to \Ty(\A, r(Q))$ is surjective. Thus
$r(Q)$ is injective in $\Ty$.

Conversely, assume $\Q=(Q_0,Q_1,\chi)$ is an injective in $\Ty$. We
claim that $Q_1$ is a divisible abelian group. In fact it suffice to
show that it is an injective object in the category $\ab$. Take any
monomorphism of abelian groups $f_1:A\to B$ and any homomorphism of
abelian groups $g_1:A_1\to Q_0$. We have to show that $g_1$ has a
lift to $B_1$. Observe that $f=(0,f_1):A[1]\to B[1]$ is faithful in
$\Ty$ and $g=(0,g_1):A[1]\to \Q$ is a well-defined morphism in
$\Ty$. By assumption there exists a morphism $\tilde{g}:B[1]\to \Q$,
It is clear that $\tilde{g}=(0,\tilde{g_1})$ for some
$\tilde{g_1}:B\to Q_1$. Thus $Q_1$ is a divisible abelian group.
Thus $r(Q_1)$ is an injective object in $\Ty$. By i) the identity
map defines a canonical morphism $i=(i_0, \id_{Q_1}):\Q\to r(Q_1)$,
which obviously is faithful in $\Ty$. Since $\Q$ is injective it
follows that there exist a morphism $q=(q_0,\id_{Q_1}):r(Q_1)\to \Q$
such that $q\circ i=\id_\Q$. Thus we have a commutative diagram
$$\xymatrix{Q_0\ar[r]^{\chi}\ar[d]^{i_0} &_2Q_1\ar[d]^{\id}\\
_2Q_1\ar[r]^{\id}\ar[d]^{q_0}& _2Q_1\ar[d]^\id\\
Q_0\ar[r]^{\chi}&_2Q_1}$$ with $q_0i_0=\id_{Q_0}$. It follows that
$i_0$ is an isomorphism. Hence $i:\Q\to r(Q_1)$ is an isomorphism
and we are done.

iii) Let $F:\S_1\to \S_2$ be a faithful morphism of symmetric
categorical groups and $G:\S_1\to K(r(Q))$ be a morphism of
symmetric categorical groups. Apply the functor $\type$ to get
morphisms $\type(F):\type(\S_1)\to \type(\S_2)$ and $\type(G):
\type(\S_1)\to r(Q))$ in $\Ty$. Since
$\pi_1(F):\pi_1(\S_1)\to\pi_1(\S_2)$ is a monomorphism of abelian
groups it follows that $\type(F):\type(\S_1)\to \type(\S_2)$ is a
faithful morphism in $\Ty$. Since $Q$ is a divisible abelian group,
$r(Q)$ is injective in $\Ty$ by ii) and we can extend $\type(F)$ to
get a morphism $\hat{g}: \type(\S_2)\to r(Q)$. Thus we have the
equality $\hat{g}\circ \type(F)=\type(G)$ in $\Ty(\type(\S_1),
r(Q))$. Since $Q=\pi_1(K(r(Q))$ is divisible abelian group the
Ext-term in the exact sequence (\ref{sinh0}) disappears and we get
the isomorphism
% the homomorphism $\pi_0(G):P\to \pi_0(\S_2)$ has
%a lifting to the homomorphism $P\to \pi_0(\S_1)$
% Since $P$ is free abelian it follows from the exact sequence
%(\ref{sinh0}) that for $i=0,1$ one has an isomorphism
\begin{equation}\label{sinh2i}\pi_0({\bf Hom}(\S_i,K(r(Q)) ))\cong
\Ty(\type(\S_1), r(Q)), \ \ i=0,1
\end{equation}
Take a morphism $L: \S_2\to K(r(Q))$ of symmetric categorical groups
which corresponds to the morphism $\hat{g}:\type(\S_2)\to r(Q)$. By
our construction one has an equality $\type(LF)=\type(G)$. This is
equality in $\pi_0({\bf Hom}(\S_1,K(r(Q))))$, which imply that the
classes of $LF$ and of $G$ in $\pi_0({\bf Hom}(\S_1,K(r(Q))))$ are
the same. Thus there exist a track from $LF$ to $G$. This shows that
$K(r(Q))$ is an injective symmetric categorical group. Conversely
assume $\S$ is an injective symmetric categorical group. Since $\S$
and $K(\type(\S))$ are equivalent, it follows that $K(\type(\S))$ is
also projective. We claim that $\type(\S)$ is injective in $\Ty$. In
fact take any faithful morphism $f=(f_0,f_1): \A\to\B$ in $\Ty$ and
any morphism $g:\A\to \type(\S)$ in $\Ty$. Then $K(f):K(\A)\to
K(\B)$ is faithful in $\scgr$. Hence for $K(g):K(\A)\to
K(\type(\S))$ we have a morphism $\tilde{G}:K(\B)\to K(\type(\S))$
and a track $ \tilde{G}\circ K(f)\to K(g)$. Thus $\tilde{G}\circ
K(f)= K(g)$ in $\pi_0({\bf Hom}(K(\type(\S)),K(\B)))$. Now apply the
functor $\type$ to get the equality $\type(\tilde{G})\circ f=g$,
showing that $\type(\S)$ is injective in $\Ty$. Hence $\type(\S)$ is
isomorphic to $r(Q)$ for a divisible abelian group $Q$. Thus $\S$
and $K(r(Q)$ are equivalent.

iv)  Let $\S$ be a symmetric categorical group. Choose a divisible
abelian group $Q$ and monomorphism of abelian groups
$f_1:\pi_1(\S)\to Q$. By Lemma \ref{adji} it has a unique extension
to a morphism $f=(f_0,f_1):\type(\S)\to r(Q)$ which is essentially
surjective. Since $Q$ is divisible abelian group, we have the
isomorphism (\ref{sinh2i}), which show that there exist a morphism
of symmetric categorical groups $ \S\to K(r(Q))$ which realizes
$f_1$ on the level of $\pi_1$ and we get the result.

v) We consider the 2-functor ${\bf Hom}(\R,-)$ from the 2-category
of symmetric categorical groups to the 2-category of categorical
right $R$-modules. It is a right 2-adjoint to the forgetful
2-functor. Since the forgetful functor is exact it follows that the
2-functor ${\bf Hom}(\R,-)$ takes injective objects to injective
ones. Let $\M$ be a categorical left $\R$-module. Choose a faithful
morphism $\M\to \Q$ in the 2-category of symmetric categorical
groups with injective symmetric categorical group $\Q$. Apply now
the 2-functor ${\bf Hom}(\R,-)$. It follows from the isomorphism
(\ref{sinh1}) that ${\bf Hom}(\R,\M)\to {\bf Hom}(\R,\Q)$ is a
faithful morphism of right $\R$-modules. By the same reasons the
obvious morphism $\M\to {\bf Hom}(\R,\M)$ is also faithful. Taking
the composite we obtain a faithful morphism $\M\to {\bf Hom}(\R,\Q)$
and hence the result.
\end{proof}

\section{On cohomology of categorical groups} Let $\G$ be a
categorical group. Recall that a \emph{$\G$-module} \cite{lh},
\cite{ulb} (or \emph{2-representation of $\G$})  is a symmetric
categorical group $\S$ together with a homomorphism of categorical
groups $\G\to {\mathcal Eq}(\S)$, where ${\mathcal Eq}(\S)$ is the
categorical group of symmetric monoidal autoequivalences of $\S$. In
the case, when $\G$ is discrete, Ulbrich \cite{ulb} defined the
cohomology groups $H^*_U(\G,\S)$. Moreover in \cite{lh},\cite{on h1}
the authors considered even more general situation when $\G$ is
arbitrary and $\S$ is assumed  to be only braided and they managed
to define the categorical groups ${\mathcal H}^i(\G,\S)$ in
dimensions $i=0$ and $i=1$.

For symmetric categorical groups one can define ${\mathcal
H}^i(\G,\S)$ for all $i$. For discrete $\G$ the connected components
of ${\mathcal H}^i(\G,\S)$ are exactly the Ulbrich's groups. The
main result of this section claims that there is an equivalence
between ${\mathcal H}^*(\G,\S)$ and appropriate ext in the abelian
2-category of $\G$-modules.

For any symmetric categorical group $\S$, we let $M(\G,\S)$ be the
symmetric categorical group of all functors from $\G$ to $\S$. In
fact $M(\G,\S)$ has a natural $\G$-module structure, induced by the
categorical group structure on $\G$. If $\S$ is a $\G$-module, then
there is a canonical morphism of $\G$-modules $$i_\S:\S\to
M(\G,\S)$$ which takes an object $a\in \S$ to the constant functor
with value $a$. Since
$$\pi_1M(\G,\S)=Maps(\pi_0(\G),\pi_1(\S))$$
we see that $i_\S$ is faithful.

For a categorical group $\G$ we let $Ner_2(\G)$ be the nerve of $\G$
as it is defined in \cite{cc1}. For a $G$-module $\S$ we let
$C^n(\G,\S)$ be the symmetric categorical group
$\prod_{Ner_2(\G)_n}\S$.
 Similarly to
the classical case, there is a 2-cochain  complex structure on
$C^*(\G,\S)$ (details can be find in \cite{car_m}). The secondary
cohomology of this complex is denoted by ${\mathcal H}^*(\G,\S)$,
while the Takeuchi-Ulbrich cohomology is denoted by $H^*_U(\G,\S)$.
By our definition the symmetric categorical groups ${\mathcal
H}^0(\G,\S)$ and ${\mathcal H}^1(\G,\S)$ coincides with one defined
in \cite{lh},\cite{on h1}, while for discrete $\G$  the groups
$H^*_U(\G,\S)$ are the same as in \cite{ulb}.

We wish to relate these objects to the secondary ext.  To do so,  we
observe that the 2-category $\scgr_\G$ of $\G$-modules  have enough
injective objects. In fact if $\Q$ is an injective object in
$\scgr$, considered as a trivial $\G$-module, then $M(\G,Q)$ is
injective in $\scgr_G$. Actually the 2-category $\scgr_\G$ of
$G$-modules has a small projective generator $M_f(\G,\Phi)$, and
hence is 2-equivalent to the 2-category of modules over a 2-ring. We
will not need this fact, therefore we omit the proof of this fact.

%here we only indicate the construction of $M_f(\G,\Phi)$. To this
%end consider $\Phi$ as a $\G$-module with trivial $G$-actions and
%consider the $\G$-submodule $M_f(\G,\Phi)$ of $M(\G,\Phi)$
%consisting of functors $f:\G\to \Phi$ such that $f(x)$ is isomorphic
%to zero for all but finite number of objects $x$ and for any
%automorphism $\al:x\to x$ the morphism $f(\al)$ is the identity
%morphism for all but finite number $\al$.

\begin{The} Let $\G$ be a categorical group and $\S$ be a $\G$-module, then
$${\mathcal H}^*(\G,\S)\cong {\bf Ext}^*_{\scgr_\G}(\Phi,\S)$$
where action of $\G$ on $\Phi$ is trivial.
\end{The}

\begin{proof} By iv) Proposition \ref{L-tv} it suffices to prove the
following assertions

i) ${\bf Hom}_{\scgr_\G}(\Phi,-)\cong {\mathcal H}^0(\G,-)$,

ii) If
$$ \xymatrix{
\S_1\ar[r]^{f}\rruppertwocell<12>^{0}{^\hskip-2ex\alpha}
&\S_2\ar[r]^{g} &\S_3}
$$
is an extension in $\scgr_\G$, then
$$ \xymatrix{
C^n(\G,\S_1)\ar[r]^{f}\rruppertwocell<12>^{0}{^\hskip-2ex\alpha}
&C^n(\G,\S_2)\ar[r]^{g} &C^n(\G,\S_3)}
$$
is also an extension.

iii) If $\S$ is an injective object in $\scgr_\G$, then ${\mathcal
H}^n(\G,\S)=0$ for $n>1$.

The assertion i) is easy consequence of the fact that $${\bf
Hom}_{\scgr}(\Phi,\S)\cong \S.$$  To see the ii) one has to use the
fact that the product of 2-exact sequences is also 2-exact and
 to show iii) one has to consider the canonical morphism $i_\S:S\to
M(\G,\S)$. Since $\S$ is injective and $i_S$ is faithful it follows
from  Corollary 4.4 \cite{ext-b-v} that $S$ is equivalent to a
direct summand of $M(\G,\S)$. Hence it suffice to show that for any
$\G$-module $\S$ one has ${\mathcal H}^n(\G,M(\G,\S))=0$ for $n>1$.
Since 'evaluation at unite' gives an explicit homotopy equivalence
between $\S$ considered as a 2-chain complex concentrated in
dimension zero and $C^*(\G,\S)$, the result follows.
\end{proof}


\begin{thebibliography}{999999999}
\bibitem{adams} {\sc J. F. Adams. } Stable homotopy and generalized homology.
Chicago Lecture Notes in Mathematics. Universityof Chicago Press.
1974.


\bibitem{butterflies} {\sc E. Aldrovandi} and {\sc B. Noohi}. Butterflies I.
Morphisms of 2-stacks. Adv. Math. 221 (2009), no. 3, 687--773.


\bibitem{bass} {\sc H. Bass}. Algebraic $K$-theory. Benjamin. 1968.

\bibitem{ah} {\sc H.-J. Baues.} Algebraic homotopy. Cambridge University Press, Cambridge,
1989.
\bibitem{htype} {\sc H.-J. Baues.} Homotopy type and homology. Oxford
Mathematical Monographs. Oxford Science Publications. The Clarendon
Press, Oxford University Press, New York, 1996.

\bibitem{mamuka_adams} {\sc H.-J. Baues} and {\sc M. Jibladze}. Secondary derived
functors and the Admas spectral sequence. Topology. 45(2006)
295-324.


\bibitem{ext-b-v} {\sc D. Bourn} and {\sc E. M. Vitale}.
Extensions of symmetric cat-groups. Homology Homotopy Appl. 4
(2002), no. 1, 103-162.

\bibitem{bcc} {\sc M. Bullejos}, {\sc P. Carrasco} and {A. M. Cegarra}.
Cohomology with coefficients in symmetric categorical groups. An
extension of Eilenberg-Mac Lane's classification theorem. Proc.
Math. Camb. Phil. Soc. 114 (1993), 163-189.

\bibitem{cc1} {\sc P. Carrasco} and {A. M. Cegarra}. (Braided) tensor structures on homotopy groupoids and nerves of (braided)
 categorical groups. Comm. Algebra. 24(1996),3995-4058.


\bibitem{cc2}  {\sc P. Carrasco} and {A. M. Cegarra}.
Schreier theory for central extensions of   categorical groups.
Comm. Algebra. 24(1996), 4059-4112.

\bibitem{cgm} {\sc P. Carrasco}, {\sc A. Garz\'on} and {\sc J. G. Miranda}. Schreier
 theory for singular extensions of categorical groups and homotopy classifications. Comm. Algebra 28(2000). 2585-2613.

\bibitem{car_m} {\sc P. Carrasco} and {\sc J. J. Martinez-Moreno},
Simplicial cohomology with coefficients in symmetric categorical
groups. Applied Categorical structures. 12 (2004) 257-285.

\bibitem{2-ch} {\sc A. del Rio}, {\sc J. Martinez-Moreno} and {\sc E. M. Vitale}. Chain
complexes of symmetric categrical groups. J. Pure and appl. Algebra,
196 (2005), 279-312.

\bibitem{dupont} {\sc M. Dupont}. Abelian categories in dimension 2.
arXiv:0809.1760.

\bibitem{lh} {\sc A. R. Garz\'{o}n} and {\sc A. del Rio}. Low-dimensional cohomology for
categorical groups. Cah. Topol. G�om. Diff�r. Cat�g. 44 (2003), no.
4, 247--280.

\bibitem{on h1} {\sc A. R. Garz\'{o}n} and {\sc A. del Rio}. On
${\mathcal H}^1$ of categorical groups. Comm. Alg. 34(2006),
3691-3699.

\bibitem{tohoku} {\sc A. Grothendieck}. Sur quelques points d'algebe homologique. Tohoku. Math. J., 9(1957), 119-221.

\bibitem{china1} {\sc F. Huang}, {\sc S. H. Chen}, {\sc W. Chen} and {\sc Z. J. Zheng}. 2-Modules
and representation of 2-rings. ArXiv: 1005.2831.

\bibitem{HS} {\sc M. J. Hopkins} and {\sc I. M. Singer}. Quadratic functions in geometry,
topology, and M-theory. J. Differential Geom. 70 (2005), no. 3,
329--452.


\bibitem{cat_rings} {\sc M. Jibladze} and {\sc T. Pirashvili}. Third Mac Lane
cohomology via categorical rings. J. Homotopy Relat. Struct., 2
(2007), pp.187-216.

\bibitem{kv} {\sc S. Kasangian} and {\sc E. M. Vitale}. Factorization systems for
symmetric cat-groups. TAC, 7 (2000), 47-70.

\bibitem{lapl} {\sc M. L. Laplaza}. Coherence for distributivity.
Lecture Notes in Math., Vol. 281, Springer, Berlin, 1972. pp.
29--65.
\bibitem{working} {\sc S. Mac Lane}. Categories for the working mathematician.
Second edition. Springer-Verlag, New York, 1998.

 \bibitem{PW}
{\sc T. Pirashvili},  {F. Waldhausen}.  Mac Lane homology and
 topological Hochschild homology. J. Pure Appl. Algebra 82 (1992), no. 1, 81-98.


\bibitem{ab-2-ab} {\sc T. Pirashvili}. Abelian categories versus abelian 2-categories.
Georgian Math. J. 16 (2009), no. 2, 353-368.


\bibitem{rpo-in} {\sc T. Pirashvili}. Projective and injective objects in symmetric categorical
groups. arXiv: 1007.0121v1.


\bibitem{2-robi} {\sc T. Pirashvili}. The derived category of
modules over a $2$-stage ring spectrum (in prepration).


\bibitem{quang} {\sc N.T. Quang}, {\sc D.D.Hanh} and {\sc N.T.Thuy}. On the
axiomatics of Ann-categories. JP J. Algebra Number Theory Appl. 11
(2008), no. 1, 59--72.

\bibitem{samson} {\sc S. Saneblidze}. On derived categories and derived functors.
Extracta Math. 22 (2007), no. 3, 315--324

\bibitem{vincent} {\sc V. Schmitt}. Enrichments over symmetric Picard
categories. arXiv:0812.0150.

\bibitem{sinh} {\sc H. X. Sinh}. $Gr$-cat\'egories, Th$\grave{e}$se de
Doctoral d'Etat. Universit\'e Paris VII, 1975.

\bibitem{tak} {\sc M. Takeuchi}. On Villamayor and Zelinsky's long exact sequence. Memoires AMS,
249(1981).

\bibitem{tak_ulb} {\sc M. Takeuchi} and {\sc K.-H. Ulbrich}. Complexes of categories
with abelian group structures. J. Pure and Appl. Algebra. 27(1983),
501-513.

\bibitem{ulb} {\sc K.-H. Ulbrich}. Group cohomology for Picard categories. J. algebra, 91 (1984), 464-498.

\bibitem{verdier} {\sc J.-L. Verdier}. Des cat\'egories d\'eriv\'ees des cat\'egories
ab\'eliennes, Ast\'erisque, 239 (1997).

\bibitem{vitale} {\sc E. M. Vitale}. A Picard-Brauer exact sequence of categorical groups.
J. Pure and Appl. Algebra. 175(2002), 283-408.
\end{thebibliography}
\end{document}